\documentclass[12pt]{article}

\usepackage[utf8]{inputenc}
\usepackage{graphicx}

\usepackage[all]{xy}
\usepackage{amssymb}
\usepackage{amsmath} 
\usepackage{amsthm}
\usepackage{verbatim}
\usepackage{latexsym}
\usepackage{mathrsfs}

\renewcommand{\labelenumi}{$\mathrm{({\roman{enumi}})}$}
\renewcommand{\emptyset}{\varnothing}

\renewcommand{\a}{\alpha}
\renewcommand{\b}{\beta}
\newcommand{\g}{\gamma}

\renewcommand{\k}{\kappa}
\renewcommand{\l}{\lambda}
\newcommand{\f}{\varphi}
\renewcommand{\phi}{\varphi}
\renewcommand{\o}{\omega}

\newcommand{\es}{\varnothing}

\newcommand{\F}{\mathcal{F}}

\newcommand{\la}{\langle} 
\newcommand{\ra}{\rangle}

\newcommand{\DLO}{{\operatorname{DLO}}}

\newcommand{\rest}{\!\restriction\!}
\newcommand{\dom}{\operatorname{dom}}

 \renewcommand{\le}{\leqslant}  
 \renewcommand{\ge}{\geqslant}  

 \newcommand{\Sii}{{\Sigma_1^1}}
 \newcommand{\Pii}{{\Pi_1^1}}

\newcommand{\reg}{\operatorname{reg}}

\newtheorem*{Thm*}{Theorem}
\newtheorem{Thm}{Theorem}[section]
\newtheorem{Lemma}[Thm]{Lemma}
\newtheorem{Prop}[Thm]{Proposition}
\newtheorem{Cor}[Thm]{Corollary}
\newtheorem{Claim}{Claim}[Thm]

\newenvironment{claim*}{\vspace{7pt}\noindent\textbf{Claim.}}{}

\theoremstyle{definition}
\newtheorem{Def}[Thm]{Definition}
\newtheorem{Question}[Thm]{Question}

\theoremstyle{remark}
\newtheorem*{Remark}{Remark}

\newcommand{\proofvpara}{\text{}}

\author{David Asper\'o, Tapani Hyttinen,\\ Vadim Kulikov, Miguel Moreno} \title{On large cardinals and generalized Baire spaces}

\usepackage[text={15cm,19cm},centering]{geometry}

\begin{document}

\maketitle
\begin{abstract}
  Working under large cardinal assumptions, we study the
  Borel-reducibility between equivalence relations modulo
  restrictions of the non-stationary ideal on some fixed cardinal
  $\kappa$. We show the consistency of
  $E^{\l^{++},\l^{++}}_{\lambda\text{-club}}$, the relation of
  equivalence modulo the non-stationary ideal restricted to
  $S^{\lambda^{++}}_\lambda$ in the space
  $(\lambda^{++})^{\lambda^{++}}$, being continuously reducible to
  $E^{2,\l^{++}}_{\lambda^+\text{-club}}$, the relation of equivalence modulo
  the non-stationary ideal restricted to
  $S^{\lambda^{++}}_{\lambda^+}$ in the space $2^{\lambda^{++}}$. Then
  we show the consistency of $E^{2,\k}_{reg}$, the relation of equivalence
  modulo the non-stationary ideal restricted to regular cardinals in
  the space $2^{\k}$, being $\Sii$-complete. We finish by showing,
  for $\Pi_2^1$-indescribable $\k$, that the isomorphism relation
  between dense linear orders of cardinality $\kappa$ is
  $\Sii$-complete.
\end{abstract}
\section{Introduction}
Throughout this article we assume that $\kappa$ is an uncountable
cardinal that satisfies $\kappa^{<\kappa}=\kappa$.  The equivalence
relations modulo (restrictions of) the non-stationary ideal have
provided a very useful tool, and a main focus of study, in generalized
descriptive set theory. In \cite{FHK} it was shown that the relation
of equivalence modulo the non-stationary ideal is not a Borel
relation, and that if $V=L$, then it is not $\Delta_1^1$. The
equivalence relation modulo the non-stationary ideal restricted to a
set stationary $S$, denoted $E^{2,\k}_S$ (see Definition~\ref{def:E2S}), is
useful when it comes to studying the complexity of the isomorphism
relations of first order theories ($\cong_T$, see
Definition~\ref{def:Isomorphism}). In \cite{FHK} it was proved that,
under some cardinality assumptions, $E^{2,\k}_{S^\k_\o}$ is Borel reducible
to $\cong_T$ for every first order stable unsuperstable theory $T$,
where $S^\k_\l$ is the set of $\l$-cofinal ordinals
below~$\k$. Similar results were obtained in \cite{FHK} for the other
non-classifiable theories. This motivates the study of the
Borel-reducibility properties of $E^{2,\k}_S$.

\begin{Thm}[\cite{FHK}, Theorem 56]
  The following is consistent: For all stationary $S$ and $S'$,
  $E^{2,\k}_S$ is Borel reducible to $E^{2,\k}_{S'}$ if and only if
  $S\subseteq S'$.
\end{Thm}
\begin{Thm}[\cite{FHK}, Theorem 55]\label{FHK-Thm55}
  The following is consistent: $E^{2,\o_2}_{S_\o^{\o_2}}$ is Borel
  reducible to $E^{2,\o_2}_{S_{\o_1}^{\o_2}}$.
\end{Thm}
In \cite{HK} the authors used the Borel-reducibility properties of the
equivalence relation modulo the non-stationary ideal to prove that in
$L$, all $\Sii$ equivalence relations are reducible to $\cong_{\DLO}$,
where $\DLO$ is the theory of dense linear orderings without end points,
which means that this equivalence relation is on top of the
Borel-reducibility hierarchy among $\Sigma^1_1$-equivalence relations,
i.e. it is $\Sigma^1_1$-complete. This result stands in contrast to
the classical, countable case, $\k=\o$, for which it is known that all
other isomorphism relation are reducible to $\cong_{\DLO}$ \cite{FS},
but far from all $\Sigma^1_1$-equivalence relations are reducible to
it; even some Borel-equivalence relations such as $E_1$ are not
reducible to any isomorphism relations in the countable case. So the
question remained: is the $\Sigma^1_1$-completeness of $\cong_{\DLO}$
just a manifestation of the pathological behaviour of $L$ or is it a
more robust property in the generalised realm. One of the
contributions of this paper is that the $\Sigma^1_1$-completeness of
$\cong_{\DLO}$ is indeed a rather robust phenomenon and holds whenever
$\k$ has certain large cardinal properties
(Theorem~\ref{thm:DLOatPi21}).

It was asked in \cite{FHK14} and in \cite[Question 3.46]{KLLS}
whether or not the equivalence relation modulo the non-stationary
ideal on the Baire space can be reduced to the Cantor space for some fixed
cofinality: in our notation, whether or not
$E^{\k,\k}_{S^\k_\mu}\le E^{2,\k}_{S^\k_\mu}$.  We approach the problem
by proving several results in this direction.  Our results have the
forms
$$E^{\k,\k}_{S^\k_\mu}\le E^{2,\k}_{S^\k_{\mu*}},$$
$$E^{\k,\k}_{S^\k_\mu}\le E^{2,\k}_{reg(\k)},$$
and
$$E^{\k,\k}_{reg(\k)}\le E^{2,\k}_{reg(\k)},$$
where $\mu^*$ is larger than $\mu$ and $\reg(\k)$ is the set of regular cardinals
below~$\k$ Mahlo. These results are obtained under various assumptions and sometimes
in forcing extensions.

Many of the results in the area of reducibility of equivalence
relations modulo non-stationary ideals use combinatorial principles,
like $\Diamond$, and other reflection principles. In this paper we
bring also some large cardinal principles into the picture.

The generalized Baire space is the set $\kappa^\kappa$ with the
bounded topology.  For every $\zeta\in \kappa^{<\kappa}$, the set
$$[\zeta]=\{\eta\in \kappa^\kappa \mid \zeta\subset \eta\}$$ 
is a basic open set. The open sets are of the form $\bigcup X$ where
$X$ is a collection of basic open sets. The collection of
$\kappa$-Borel subsets of $\kappa^\kappa$ is the smallest set which
contains the basic open sets and is closed under unions and
intersections of length $\kappa$. Since in this paper we do not consider any other
kind of Borel sets besides $\k$-Borel, we will omit the prefix ``$\kappa$-''.\\

The generalized Cantor space is the subspace $2^\kappa\subset \k^\k$
with the relative subspace to\-po\-lo\-gy. For $X,Y\in \{\k^\k,
2^\k\}$, we say that a function $f\colon X\rightarrow Y$ is
\emph{Borel} if for every open set $A\subseteq Y$ the inverse image
$f^{-1}[A]$ is a Borel subset of $X$. Let $E_1$ and $E_2$ be
equi\-va\-lence relations on $X$ and $Y$ respectively. We say that
$E_1$ is \emph{Borel reducible} to $E_2$ if there is a Borel function
$f\colon X\rightarrow Y$ that satisfies $(\eta,\xi)\in
E_1\Leftrightarrow (f(\eta),f(\xi))\in E_2$.  We call $f$ a
\emph{reduction} of $E_1$ to $E_2$. This is denoted by $E_1\le_B E_2$,
and if $f$ is continuous,
then we say that $E_1$ is \emph{continuously reducible} to $E_2$, which is denoted by $E_1\le_c E_2$.\\

For every stationary $S\subset \kappa$, we define the equivalence
relation modulo the non-stationary ideal restricted to a stationary set $S$, on the
space $\l^\k$ for $\l\in\{2,\k\}$:

\begin{Def} \label{def:E2S} For every stationary $S\subset \kappa$ and
  $\l\in\{2,\k\}$, we define $E^{\l,\k}_S$ as the
  relation $$E^{\l,\k}_S=\{(\eta,\xi)\in \l^\k\times \l^\k\mid
  \{\a<\k\mid \eta(\a)\ne \xi(\a)\}\cap S \text{ is not
    stationary}\}.$$
\end{Def}

Note that $E^{2,\k}_S$ can be identified with the equivalence relation on
the power set of $\kappa$ in which two sets $A$ and $B$ are equivalent
if their symmetric difference restricted to $S$ is non-stationary.
This can be done by identifying a set $A\subset \k$ with its 
characteristic function.

For every regular cardinal $\mu<\kappa$, we denote $\{\alpha<\kappa
\mid cf(\alpha)=\mu\}$ by $S_\mu^\kappa$.  A set $C$ is
$\mu$-\emph{club} if it is unbounded and closed under $\mu$-limits.
For brevity, when $S=S^\k_\mu$, we will denote $E^{\l,\k}_{S^\k_\mu}$
by $E^{\l,\k}_{\mu\text{-club}}$. Note that $(f,g)\in
E^{\l,\k}_{\mu\text{-club}}$ if and only if the set $\{\alpha<\k\mid
f(\a)=g(\a)\}$ contains a $\mu$-club. 

For a Mahlo cardinal $\k$, the set $reg(\k)=\{\a<\k\mid\a\text{ a
  regular cardinal}\}$ is stationary. We will denote the equivalence
relation $E^{\l,\k}_{\reg(\k)}$ by $E^{\l,\k}_{reg}$. 

Given an equivalence relation $E$ on $X\in \{\k^\k,2^\k\}$, we can
define the $\lambda$-product relation of $E$ for any
$0<\lambda<\kappa$. The $\lambda$-product relation
$\mathbf{\Pi}_{\lambda}E$ is the relation defined on $X^\lambda\times
X^\lambda$ by $\eta\ \mathbf{\Pi}_{\lambda}E\ \xi$ if $\eta_\gamma\ E\
\xi_\gamma$ holds for every $\gamma<\lambda$, where
$\eta=(\eta_\gamma)_{\gamma<\lambda}$ and
$\xi=(\xi_\gamma)_{\gamma<\lambda}$. We endow the space $X^\lambda$,
$X\in \{\k^\k,2^\k\}$, with the box topology generated by the basic
open sets:
$$\{\mathbf{\Pi}_{\alpha<\lambda}\mathcal{O}_\alpha\mid\forall\alpha<\lambda\,(\mathcal{O}_\alpha\text{ is an open set in }X)\}.$$

One of the motivations to study Borel reducibility in generalized
Baire spaces is the connection with model theory. This connection
consists in the possibility to study the Borel reducibility of the
isomorphism relation of theories by coding structures with universe
$\k$ via elements of $\k^\k$. We may fix this coding, relative to a
given countable relational vocabulary $\mathcal{L}=\{P_n\mid
n<\omega\}$, as in the following definition.

\begin{Def}
  Fix a bijection $\pi\colon \k^{<\o}\to \k$. For every $\eta\in
  \k^\k$ define the $\mathcal{L}$-structure $\mathcal{A}_\eta$ with
  universe $\k$ as follows: For every relation $P_m$ with arity $n$,
  every tuple $(a_1,a_2,\ldots , a_n)$ in $\k^n$ satisfies
  $$(a_1,a_2,\ldots , a_n)\in P_m^{\mathcal{A}_\eta}\Longleftrightarrow \eta(\pi(m,a_1,a_2,\ldots,a_n)) \ge 1.$$
\end{Def}
When we describe a complete theory $T$ in a vocabulary
$\mathcal{L}'\subseteq \mathcal{L}$, we think of it as a complete
$\mathcal{L}$-theory extending $T\cup \{\forall \bar{x}\neg
P_n(\bar{x})\,|\,P_n\in \mathcal{L}\backslash \mathcal{L}'\}$.

\begin{Def}[The isomorphism relation] 
  \label{def:Isomorphism}
  Assume $T$ is a complete first
  order theory in a countable vocabulary. We define $\cong_T$ as the
  relation $$\{(\eta,\xi)\in \k^\kappa\times \k^\kappa\mid
  (\mathcal{A}_\eta\models T, \mathcal{A}_\xi\models T,
  \mathcal{A}_\eta\cong \mathcal{A}_\xi)\text{ or }
  (\mathcal{A}_\eta\not\models T, \mathcal{A}_\xi\not\models T)\}.$$
\end{Def}

In the second section we will study the reducibility between different
cofinalities, and in the last section we will study the reducibility
of $E^{\k,\k}_{reg}$ and $E^{2,\k}_{reg}$. Here is the list of the
main results in this article:
\begin{itemize}
\item (Theorem \ref{indescr-forcing}) \textit{Suppose $\kappa$ is a
    $\Pi_1^{\lambda^+}$-indescribable cardinal for some $\l<\k$ and
    $V=L$. Then there is a forcing extension where $\kappa$ is
    collapsed to become $\lambda^{++}$ and $E_{\lambda\text{-club}}^
    {\l^{++},\l^{++}}\ \leq_c\ E_{\lambda^+\text{-club}}^{2,\l^{++}}$.}
\item (Corollary 2.14) \textit{ The following statement is consistent
    relative to the consistency of countably many supercompact cardinals:
    $E^{\omega_2,\omega_2}_{\omega\text{-club}}\leq_c
    E^{\omega_2,\omega_2}_{\omega_1\text{-club}}$, and for every $n>2$ and
    every $0\leq k\leq n-3$,
    $E^{\omega_n,\omega_n}_{\omega_k\text{-club}}\leq_c
    E^{\omega_n,\omega_n}_{\omega_{n-1}\text{-club}}$.}
    
  This corollary follows from [\cite{JS}, Theorem 1.3] and gives a
  model (different from $L$ or the one in Theorem \ref{FHK-Thm55}) in
  which reducibility between different cofinalities holds.
\item (Theorem 3.3) \textit{Suppose $S=S_\lambda^\k$ for some regular
    cardinal $\lambda<\k$, or $S=reg(\kappa)$ and $\k$ weakly
    compact. If $\k$ has the $S$--dual diamond (Definition~\ref{S-Dual
      Diamond}), then $E^{\k,\k}_{S}\leq_c E^{2,\k}_{reg}$.}
\item (Corollary 3.5) \textit{Suppose $V=L$ and $\k$ is weakly
    compact. Then $E^{2,\k}_{reg}$ is $\Sii$-complete.}
\item (Theorem 3.6) \textit{Suppose $\k$ is a supercompact
    cardinal. There is a generic extension $V[G]$ in which
    $E_{reg}^{\k,\k}\leq_c E^{2,\k}_{reg}$ and $\k$ is still supercompact
    in the extension.}
\item (Theorem 3.7) \textit{If $\k$ is a $\Pi_2^1$-indescribable
    cardinal, then $E^{\k,\k}_{reg}$ is $\Sigma_1^1$-complete.}
\item (Corollary 3.8) \textit{Suppose $\k$ is a supercompact cardinal. 
    There is a generic extension $V[G]$ in which
    $\k$ is still supercompact and $E_{reg}^{2,\k}$ 
    is $\Sii$-complete.}
\item (Theorem 3.9) \textit{Let DLO be the theory of dense linear
    orderings without end points. If $\k$ is a
    $\Pi_2^1$-indescribable cardinal, then $\cong_{\DLO}$ is
    $\Sii$-complete.}
\end{itemize}

\section{Reducibility between different cofinalities}
In \cite{FHK} the authors studied the reducibility between the
relations $E^{2,\k}_{\mu\text{-club}}$ and showed in particular the
consistency of $E_{\lambda\text{-club}}^{2,\l^{++}}\leq_c
E_{\lambda^+\text{-club}}^{2,\l^{++}}$. In this section we continue
along these lines.
\begin{Def}\label{def:StronglyReflects}
  We say that a set $X\subset\k$ strongly reflects to a set $Y\subset
  \k$ if for all stationary $Z\subset X$ there exist stationary many
  $\alpha\in Y$ with $Z\cap\alpha$ stationary in $\alpha$.
\end{Def}
In \cite[Theorem 55]{FHK} it is proved that: If $\kappa$ is a weakly
compact cardinal, then $S_\lambda^\kappa$ strongly reflects to $reg
(\kappa)$, for any regular cardinal $\lambda<\kappa$. This result can
be generalized to $\Pi^\lambda_1$-indescribable cardinals:

\begin{Def}\label{def:PiLambda1Indescr}
  A cardinal $\kappa$ is $\Pi^\lambda_1$-indescribable if whenever
  $A\subset V_\kappa$ and $\sigma$ is a $\Pi_1$ sentence such that
  $(V_{\kappa+\lambda}, \in, A)\models \sigma$, then for some
  $\alpha<\kappa$, $(V_{\alpha+\lambda}, \in, A\cap V_{\alpha})\models
  \sigma$.
\end{Def}

Strongly unfoldable cardinals are examples of
$\Pi^\lambda_1$-indescribable cardinals.

\begin{Lemma}\label{indescr-strong-reflection}
  Suppose $\kappa$ is a $\Pi^\lambda_1$-indescribable cardinal. There
  are $\lambda$ many disjoint stationary subsets of $\kappa$, $\la
  S_\gamma\ra_{\gamma<\lambda}$, such that for every $\gamma<\lambda$,
  $S_\gamma\subseteq reg(\kappa)$ and $\kappa$ strongly reflects to
  $S_\gamma$.
\end{Lemma}
\begin{proof}
  Let $S^*_\beta$ denote the set of all the
  $\Pi_1^\beta$-indescribable cardinals below $\kappa$.  Since
  ``$\kappa$ is $\Pi^\beta_1$-indescribable'' is a $\Pi^1_1$ property
  of the structure $(V_{\kappa+\lambda},\in)$, the set $S^*_\beta$ is
  stationary for every $\beta<\lambda$.

  Let us show that for every stationary set $X\subseteq
  \k$, $$B=\{\alpha\in S^*_\beta \mid X\cap \alpha \mbox{ is
    stationary in } \alpha\}$$ is stationary.  Let $C$ be a club in
  $\kappa$. The sentence $$(C \mbox{ is unbounded in }\kappa)\wedge (X
  \mbox{ is stationary in } \kappa)\wedge (\kappa \mbox{ is }
  \Pi^\beta_1\mbox{-indescribable})$$ is a $\Pi^1_1$ property of the
  structure $(V_{\kappa+\lambda}, \in, X, C)$. By reflection, there is
  $\gamma<\kappa$ such that $C\cap \gamma$ is unbounded in $\gamma$,
  and hence $\gamma\in C$, $S\cap \gamma$ is stationary in $\gamma$,
  and $\gamma$ is $\Pi^\beta_1$-indescribable. We conclude that
  $C\cap B\neq \emptyset$.

  Let us denote $S_\beta^*\backslash S_{\beta+1}^*$ by $S_\beta$. Let
  us show that for every stationary set $X\subseteq \k$, $$\{\alpha\in
  S_\beta \mid X\cap \alpha \text{ is stationary in } \alpha\}$$ is
  stationary. Let $C$ be a club in $\kappa$. Since $\{\alpha\in
  S^*_\beta \mid X\cap \alpha \text{ is stationary in } \alpha\}$ is
  stationary, we can pick $\gamma\in C\cap \{\alpha\in S^*_\beta \mid
  X\cap \alpha \text{ is stationary in } \alpha\}$ such that $\gamma$
  is minimal.

  \begin{Claim}
    $\gamma$ is not $\Pi_1^{\beta+1}$-indescribable.
  \end{Claim}
  \begin{proof}
    Suppose, towards a contradiction, that $\gamma$ is
    $\Pi^{\beta+1}_1$-indescribable. The sentence
    $$(C\cap \gamma \mbox{ is unbounded in }\gamma)\wedge (X\cap \gamma \mbox{ is stationary in }\gamma)\wedge (\gamma \mbox{ is }\Pi^\beta_1\mbox{-indescribable})$$ 
    is a $\Pi^1_1$ property of the structure $(V_{\gamma+\beta+1},
    \in, X\cap\gamma, C\cap\gamma)$. By reflection, there is
    $\gamma'<\gamma$ such that $C\cap\gamma'$ is unbounded in
    $\gamma'$, $X\cap \gamma'$ is stationary in $\gamma'$, and
    $\gamma'$ is $\Pi^\beta_1$-indescribable. This contradicts the
    minimality of $\gamma$.
  \end{proof}
  We conclude that $S_\beta$ is stationary and $\{\alpha\in S_\beta
  \mid X\cap \alpha \text{ is stationary in } \alpha\}$ is stationary,
  for every $\beta<\lambda$.
\end{proof}

The notion of $\diamond$-reflection was introduced in \cite{FHK} in
order to find reductions between equivalence relations modulo
non-stationary ideals (see below).

\begin{Def}[$\diamond$-reflection]
  Let $X,Y$ be subsets of $\kappa$ and suppose $Y$ consists of
  ordinals of uncountable cofinality. We say that \textit{$X$
    $\diamond$-reflects to $Y$} if there exists a sequence $\la
  D_\alpha\ra_{\alpha\in Y}$ such that:
  \begin{itemize}
  \item $D_\alpha\subset \alpha$ is stationary in $\alpha$ for all
    $\alpha\in Y$.
  \item if $Z\subset X$ is stationary, then $\{\alpha\in Y\mid
    D_\alpha=Z\cap\alpha\}$ is stationary.
  \end{itemize}
\end{Def}
\begin{Thm}[\cite{FHK}, Theorem 59]
  Suppose $V=L$ and that $X\subseteq \kappa$ and $Y\subseteq reg
  (\kappa)$. If $X$ strongly reflects to $Y$, then $X$
  $\diamond$-reflects to $Y$.
\end{Thm}

\begin{Thm}[\cite{FHK}, Theorem 58] \label{FHK-Thm58} If $X$
  $\diamond$-reflects to $Y$, then $E^{2,\k}_X\leq_c E^{2,\k}_Y$.
\end{Thm}

$\diamond$-reflection also implies some reductions for the relations
$E_{\mu\text{-club}}^\k$ on the space $\k^\k$. To show this, we first
need to introduce some definitions.

\begin{Def}
For every $\alpha<\k$ with $\gamma<cf(\a)$ define $E_{\gamma\text{-club}}^{\k,\k}\restriction\alpha$ by:
$$E_{\gamma\text{-club}}^{\k,\k}\restriction\alpha=
\{(\eta,\xi)\in \kappa^\kappa\times\kappa^\kappa\mid \exists C\subseteq \alpha\text{ a  } \gamma\text{-club}, \forall\beta\in C, \eta(\beta)=\xi(\beta)\}.$$
\end{Def}
\begin{Prop}\label{refl-red}
  Suppose $\gamma<\lambda<\k$ are regular cardinals. If $S_\gamma^\kappa$
  strongly reflects to $S_\lambda^\k$, then
  $E_{\gamma\text{-club}}^{\k,\k}\leq_c
  E_{\lambda\text{-club}}^{\k,\k}$.
\end{Prop}

\begin{proof}
  Suppose that for every stationary set $S\subset S_{\gamma}^{\kappa}$
  it holds that $\{\a\in S_\lambda^\k\mid S\cap \a\text{ is stationary
    in }\a\}$ is a stationary set, and define $F\colon \k^\k\to \k^\k$
  by
  $$F(\eta)(\a)=
  \begin{cases}
    f_\alpha(\eta), \text{ if }cf(\a)=\lambda\\
    0, \text{ otherwise}.
  \end{cases}
  $$
  where $f_\eta(\alpha)$ is a code in $\kappa\backslash \{0\}$ for the
  $(E_{\gamma\text{-club}}^{\k,\k}\restriction\alpha)$-equivalence
  class of~$\eta$.

  Let us prove that if $(\eta,\xi)\in E^{\k}_{\gamma\text{-club}}$,
  then $(F(\eta),F(\xi))\in E^{\k,\k}_{\lambda\text{-club}}$.  Suppose
  $(\eta,\xi)\in E^{\k}_{\gamma\text{-club}}$. There is a
  $\gamma$-club where $\eta$ and $\xi$ coincide and so there is a
  club $C$ such that for all $\a\in C\cap S^{\k}_{\lambda}$ the
  functions $\eta$ and $\xi$ are
  $(E^{\k}_{\gamma\text{-club}}\rest_\a)$-equivalent. Thus, by the
  definition of $F$, for all $\a\in C\cap S^{\k}_{\lambda}$,
  $F(\eta)(\a)=F(\xi)(\a)$. We conclude that $(F(\eta),F(\xi))\in
  E^{\k,\k}_{\lambda\text{-club}}$.

  Let us prove that if $(\eta,\xi)\notin E^{\k}_{\gamma\text{-club}}$,
  then $(F(\eta),F(\xi))\notin E^{\k,\k}_{\lambda\text{-club}}$. Suppose
  that $(\eta,\xi)\notin E^{\k}_{\gamma\text{-club}}$.  Then there is
  a stationary $S\subset S^{\k}_\gamma$ on which $\eta(\a)\ne\xi(\a)$.
  Since $A=\{\a\in S_\lambda^\k\mid S\cap \a\text{ is stationary in
  }\a\}$ is a stationary and for all $\a\in A$, $f_\a(\eta)\neq
  f_\a(\xi)$, we conclude that $(F(\eta),F(\xi))\notin
  E^{\k,\k}_{\lambda\text{-club}}$.
\end{proof}

\begin{Cor}
  Suppose $\gamma<\lambda<\k$ are regular cardinals. If $S_\gamma^\kappa$
  $\diamond$-reflects to $S_\lambda^\kappa$, then
  \begin{enumerate}
  \item $E_{\gamma\text{-club}}^{2,\k}\leq_c E_{\lambda\text{-club}}^{2,\k}$.
  \item $E_{\gamma\text{-club}}^{\k,\k}\leq_c E_{\lambda\text{-club}}^{\k,\k}$.
  \end{enumerate}
\end{Cor}

\begin{proof}
  \begin{enumerate}
  \item Follows from Theorem \ref{FHK-Thm58}.
  \item By the definition of $\diamond$-reflection, $S_\gamma^\kappa$
    $\diamond$-reflecting to $S_\lambda^\kappa$ implies that for all
    $S\subseteq S_\gamma^\kappa$ the set $\{\a\in S_\lambda^\k\mid
    S\cap \a\text{ is stationary in }\a\}$ is a stationary set. The
    result follows from Proposition \ref{refl-red}.
  \end{enumerate}
\end{proof}

In \cite{FHK}, the consistency of $S_\lambda^{\lambda^{++}}$
$\diamond$-reflecting to $S_{\lambda^+}^{\lambda^{++}}$ was
shown. This gives a model in which $E_{\lambda\text{-club}}^{2,\k}\leq_c
E_{\lambda^+\text{-club}}^{2,\k}$ and
$E_{\lambda\text{-club}}^{\lambda^{++}}\leq_c
E_{\lambda^+\text{-club}}^{\lambda^{++}}$.

\begin{Thm}[\cite{FHK}, Theorem 55]\label{[FHK]-Theorem-55} 
  Suppose that $\kappa$ is a weakly compact cardinal and $V=L$. Then:
  \begin{enumerate}
  \item $E_{\lambda\text{-club}}^{2,\k}\leq_c E_{reg}^{2,\k}$ holds for all
    regular $\lambda<\kappa$.
  \item For every regular $\lambda<\kappa$ there is a forcing
    extension where $\kappa$ is collapsed to become $\lambda^{++}$ and
    $E_{\lambda\text{-club}}^{2,\l^{++}}\leq_c E_{\lambda^+\text{-club}}^{2,\l^{++}}$.
  \end{enumerate}
\end{Thm}

The proof of this theorem can be generalised using
Lemma~\ref{indescr-strong-reflection} to show the consistency
of $E_{\l\text{-club}}^{\l^{++},\l^{++}}\ \leq_c\
E_{\l^+\text{-club}}^{2,\l^{++}}$:

\begin{Thm}\label{indescr-forcing}
  Suppose $\kappa$ is a $\Pi_1^{\lambda^+}$-indescribable cardinal
  and that $V=L$. Then there is a forcing extension where $\kappa$ is
  collapsed to become $\lambda^{++}$ and $E_{\lambda\text{-club}}^
  {\lambda^{++}}\ \leq_c\ E_{\lambda^+\text{-club}}^ {2}.$
\end{Thm}
\begin{proof}
  Let us collapse $\kappa$ to $\lambda^{++}$ with the Levy
  collapse $$\mathbb{P}=\{f:reg(\kappa)\rightarrow\kappa^{<\lambda^+}\,\mid\,
  rang(f(\mu))\subset \mu,\arrowvert\{\mu\mid
  f(\mu)\neq\emptyset\}\arrowvert\leq\lambda\}$$ where $f\geq g$ if
  and only if $f(\mu)\subseteq g(\mu)$ for all $\mu\in
  reg(\kappa)$. Let us define $\mathbb{P}_\mu$ and $\mathbb{P}^\mu$
  for all $\mu$ by: $\mathbb{P}_\mu=\{f\in\mathbb{P}\mid
  sprt(f)\subset \mu\}$ and $\mathbb{P}^\mu=\{f\in\mathbb{P}\mid
  sprt(f)\subset \kappa\backslash \mu\}$. It is known that all regular
  $\lambda<\mu\leq \kappa$ satisfy:
  \begin{enumerate}
\item if $\mu>\lambda^+$, then $\mathbb{P}_\mu$ has the $\mu$-c.c.,
\item $\mathbb{P}_\mu$ and $\mathbb{P}^\mu$ are ${<}\lambda^+$-closed,
\item $\mathbb{P}=\mathbb{P}_\kappa\Vdash \lambda^{++}=\check{\kappa}$,
\item if $\mu<\kappa$, then $\mathbb{P}\Vdash
  cf(\check{\mu})=\lambda^+$,
\item if $p\in\mathbb{P}$, $\sigma$ a name, and $p\Vdash\sigma\mbox{
    ``is a club in }\lambda^{++}$'', then there is a club $E\subset
  \kappa$ such that $p\Vdash\check{E}\subset \sigma$.
\end{enumerate}
  \begin{Claim}
    There is a sequence $\la S_\gamma\ra_{\gamma<\lambda^+}$ of
    disjoint stationary subsets of $S^{\lambda^{++}}_{\lambda^+}$ in
    $V[G]$ such that $S_{\lambda}^{\lambda^{++}}$ $\diamond$-reflects
    to $S_\gamma$ for every $\gamma<\lambda^+$.
  \end{Claim}
  \begin{proof}
    Let $G$ be a $\mathbb{P}$-generic over $V$, and define
    $G_\mu=G\cap \mathbb{P}_\mu$ and $G^\mu=G\cap \mathbb{P}^\mu$. So
    $G_\mu$ is $\mathbb{P}_\mu$-generic over $V$, $G^\mu$ is
    $\mathbb{P}^\mu$-generic over $V[G_\mu]$, and
    $V[G]=V[G_\mu][G^\mu]$.  Let $S^*_\beta$ denote the set of all the
    $\Pi_1^\beta$-indescribable cardinals below $\kappa$ and
    $S_\beta=S^*_\beta\backslash S^*_{\beta+1}$. We will show that
    $S_\lambda^{\lambda^{++}}$ $\diamond$-reflects to $\{\mu\in
    V[G]\mid \mu\in S_\beta^V\}$ for all $\beta<\lambda^+$. Let us fix
    $\beta<\lambda^+$ and denote by $Y$ the set $\{\mu\in V[G]\mid
    \mu\in S_\beta^V\}$. By Lemma \ref{indescr-strong-reflection} we
    know that $S_\beta^V$ is stationary and by (v), it remains
    stationary in $V[G]$. By (i) we know that there are no antichains
    of length $\mu$ in $\mathbb{P}_\mu$, and since
    $\arrowvert\mathbb{P}_\mu\arrowvert=\mu$ we conclude that there
    are at most $\mu$ antichains. On the other hand, there are $\mu^+$
    many subsets of $\mu$. Hence, there is a
    bijection $$h_\mu:\mu^+\rightarrow\{\sigma\mid \sigma \mbox{ is a
      nice }\mathbb{P}_\mu\mbox{ name for a subset of }\mu\}$$ for
    each $\mu\in reg(\kappa)$ such that $\mu>\lambda^{+}$, where a nice
    $\mathbb{P}_\mu$ name for a subset of $\check{\mu}$ is of the form
    $\bigcup\{\{\check{\alpha}\}\times A_\alpha\mid\alpha\in B\}$ with
    $B\subset \check{\mu}$ and $A_\alpha$ an antichain in
    $\mathbb{P}_\mu$. Notice that the nice $\mathbb{P}_\mu$ names for
    subsets of $\check{\mu}$ are subsets of $V_\mu$. Let us
    define $$D_\mu=
  \begin{cases} 
    [h_\mu([(\cup G)(\mu^+)](0))]_G & \mbox{ if this set is stationary }\\
    \mu & \mbox{otherwise. }
  \end{cases}$$ We will show that $\la D_\mu\ra_{\mu\in Y}$ is the
  needed $\diamond$-sequence in $V[G]$.

  Suppose, towards a contradiction, that there are a stationary set
  $S\subset S_\lambda^{\lambda^{++}}$ and a club $C\subset
  \lambda^{++}$ (in $V[G]$) such that for all $\alpha\in C\cap Y$,
  $D_\alpha\neq S\cap \alpha$. By (v) there is a club $C_0\subset C$
  such that $C_0\in V$. Let $\dot{S}$ be a nice name for $S$ and $p$ a
  condition such that $p$ forces that $\dot{S}$ is stationary. We will
  show that $$H=\{q<p\mid q\Vdash D_\mu=\dot{S}\cap\check{\mu}\mbox{
    for some }\mu\in C_0\}$$ is dense below $p$, which is a
  contradiction. Let us redefine $\mathbb{P}$. Let
  $\mathbb{P}^*=\{q\,\mid\, \exists r\in\mathbb{P}\,(r\restriction
  sprt (r)=q)\}$. Clearly $\mathbb{P}\cong\mathbb{P}^*$, $\mathbb{P}^*
  \subseteq V_\kappa$, and $\mathbb{P}_\mu^*=\mathbb{P}^*\cap V_\mu$,
  where $\mathbb{P}_\mu^*=\{q\,\mid\, \exists r\in\mathbb P_\mu\, (r
  \restriction sprt(r) = q)\}$. It can be verified that the properties
  mentioned above also hold for $\mathbb{P}_\mu^*$. From now on denote
  $\mathbb{P}_\mu^*$ by $\mathbb{P}_\mu$. Let $r$ be a condition
  stronger than $p$
  and $$R=(\mathbb{P}\times\{0\})\cup(\dot{S}\times\{1\})\cup(C_0\times
  \{2\}\cup(\{r\}\times \{3\})).$$ Let $\forall A\varphi$ be the
  formula:

  \textit{If $A$ is closed and unbounded and $t<r$ are arbitrary, then
    there exists $q<r$ and $\a\in A$ such that
    $q\Vdash_\mathbb{P}\check{\a}\in \dot{S}$}.

  Clearly, $\forall A\varphi$ says $r\Vdash (\dot{S}\mbox{ is
    stationary})$. By (v) it is enough to quantify over club sets in
  $V$. Notice that $t<r$, $q<t$, $A$ is a club, and $\alpha\in A$ are
  first order expressible using $R$ as a parameter. The definition of
  $\check{\a}$ is recursive in
  $\alpha$: $$\check{\a}=\{(\check{\gamma},1_\mathbb{P})\mid
  \gamma<\a\}$$ and it is absolute for $V_\kappa$. Then
  $q\Vdash_\mathbb{P}\check{\a}\in \dot{S}$ is equivalent to saying
  that for each $q'<q$ there exists $q''<q'$ with $(\check{\a},q'')\in
  \dot{S}$, and this is first order expressible using $R$ as a
  parameter. Therefore $\forall A\varphi$ is a $\Pi^1_1$ property of
  the structure $(V_\kappa,\in,R)$, even more $$(\forall
  A\varphi)\wedge (\kappa \mbox{ is
  }\Pi^\beta_1\mbox{-indescribable})$$ is a $\Pi^1_1$ property of the
  structure $(V_{\kappa+\lambda}, \in, R)$. By reflection, there is
  $\mu<\kappa$ $\Pi^\beta_1$-indescribable, such that $\mu\in C_0$, $r\in \mathbb{P}_\mu$, and
  $(V_\mu,\in,R)\models \forall A\varphi$. In the same way as in Claim 2.3.1, we can show that there is there is
  $\mu<\kappa$ $\Pi^\beta_1$-indescribable that is not $\Pi^{\beta+1}_1$-indescribable, i.e. $(\check{\mu}_G\in
  Y)^{V[G]}$, such that $\mu\in C_0$, $r\in \mathbb{P}_\mu$, and
  $(V_\mu,\in,R)\models \forall A\varphi$.
  Notice that $\alpha\in
  S\cap\mu$ implies that $(\check{\a},\check{q})\in \dot{S}$ for some
  $q\in \mathbb{P}_\mu$. Let $\dot{S}_\mu=\dot{S}\cap V_\mu$, thus
  $r\Vdash_{\mathbb{P}_\mu} (\dot{S}_\mu\mbox{ is stationary})$. Let
  us define $q$ as follows: $dom(q)=dom(r)\cup\{\mu^+\}$,
  $q\restriction \mu=r\restriction\mu$ and $q(\mu^+)=f$,
  $dom(f)=\{0\}$, and $f(0)=h_\mu^{-1}(\dot{S}_\mu)$. Since
  $\mathbb{P}^\mu$ is ${<}\lambda^+$-closed and does not kill
  stationary subsets of $S_\lambda^{\lambda^{++}}$,
  $(\dot{S}_\mu)_{G_\mu}$ is stationary in $V[G]$, and by the way we
  chose $\mu$, $(\dot{S}_\mu)_{G_\mu}=(\dot{S}_\mu)_{G}$. Therefore
  $q\Vdash_\mathbb{P}(\dot{S}_\mu\mbox{ is stationary})$, and by the
  definition of $D_\mu$ (in $V[G]$) we conclude that
  $q\Vdash_\mathbb{P}\dot{S}_\mu=D_\mu$. Finally, by the way we chose
  $\mu$, we get that $(\dot{S}_\mu)_{G}=S\cap\mu$. We conclude that
  $H$ is dense below $p$, a contradiction.
\end{proof}
From now on in this proof, we will work in $V[G]$. In particular,
$\kappa$ will be $\lambda^{++}$.
\begin{Claim}
  $E_{\lambda\text{-club}}^{\k,\k}\ \leq_c \mathbf{\Pi}_{\lambda^+}\
  E_{\lambda\text{-club}}^{2,\k}$.
\end{Claim}
\begin{proof}
  Let $H$ be a bijection from $\kappa$ to $2^{\lambda^+}$. Define
  $\mathcal{F}:\kappa^\kappa\rightarrow(2^\kappa)^{\lambda^+}$ by
  $\mathcal{F}(f)=(f_\gamma)_{\gamma<\lambda^+}$, where
  $f_\gamma(\alpha)=H(f(\alpha))(\gamma)$ for every $\gamma<\lambda^+$
  and $\alpha<\kappa$. Let us show that $\mathcal{F}$ is a reduction
  of $E^{\k,\k}_{\lambda\text{-club}}$ to $\mathbf{\Pi}_{\lambda^+}\
  E^{2,\k}_{\lambda\text{-club}}$.
  
  Clearly $f(\alpha)=g(\alpha)$ implies $H(f(\alpha))=H(g(\alpha))$
  and $f_\gamma(\alpha)=g_\gamma(\alpha)$ for every
  $\gamma<\lambda^+$. Therefore, $f\ E_{\lambda\text{-club}}^{\k,\k}\
  g$ implies that for all $\gamma<\lambda^+$, $f_\gamma\
  E_{\lambda\text{-club}}^{2,\k}\ g_\gamma$ holds. So $f\
  \mathbf{\Pi}_{\lambda^+}\ E_{\lambda\text{-club}}^{2,\k}\ g$.
  
  Suppose that for every $\gamma<\lambda^+$ there is $C_\gamma$, a
  $\lambda$-club, such that $f_\gamma(\alpha)=g_\gamma(\alpha)$ holds
  for every $\alpha\in C_\gamma$. Since the intersection of less than
  $\kappa$ $\lambda$-club sets is a $\lambda$-club set, there is a
  $\lambda$-club $C$ on which the functions $f_\gamma$ and $g_\gamma$
  coincide for every $\gamma<\lambda^+$. Therefore
  $H(f(\alpha))(\gamma)=H(g(\alpha))(\gamma)$ holds for every
  $\gamma<\lambda^+$ and every $\alpha\in C$, so
  $H(f(\alpha))=H(g(\alpha))$ for every $\alpha\in C$. Since $H$ is a
  bijection, we can conclude that $f(\alpha)=g(\alpha)$ for every
  $\alpha\in C$, and hence $f\ E_{\lambda\text{-club}}^{\k,\k}\ g$.
\end{proof}
By Claim 2.11.1, there is a sequence $\la
S_\gamma\ra_{\gamma<\lambda^+}$ of disjoint stationary subsets of
$S^{\kappa}_{\lambda^+}$ such that $S_\lambda^\kappa$
$\diamond$-reflects to $S_\gamma$ for all $\gamma<\lambda^+$. Let $\la
D_\alpha^\gamma\ra_{\alpha\in S_\gamma}$ be a sequence that witnesses
that $S_\lambda^\kappa$ $\diamond$-reflects to $S_\gamma$.

For every $\eta\in \kappa^\kappa$ define $F(\eta)$ by:
$$F(\eta)(\alpha)=
\begin{cases}
  1 &\mbox{if there is  }\gamma<\lambda^+ \mbox{ with }\alpha\in S_\gamma \mbox{ and } \F(\eta)_\gamma^{-1}[1]\cap D_\alpha^\gamma\mbox{ stationary in }\alpha\\
  0 & \mbox{otherwise }
\end{cases}
$$
where $(\F(\eta)_\gamma)_{\gamma<\lambda^+}=\mathcal{F}(\eta)$
and where $\mathcal{F}$ is the reduction given by Claim 2.11.2.

Suppose $\eta$, $\xi$ are not
$E^{\k,\k}_{\lambda\text{-club}}$-equivalent. By Claim 2.11.2 there
exists $\gamma<\lambda^+$ such that
$\F(\eta)_\gamma^{-1}[1]\Delta\F(\xi)_\gamma^{-1}[1]$ is
stationary. Therefore, either $\F(\eta)_\gamma^{-1}[1]\backslash
\F(\xi)_\gamma^{-1}[1]$ or $\F(\xi)_\gamma^{-1}[1]\backslash
\F(\eta)_\gamma^{-1}[1]$ is stationary. Without loss of generality, let us
assume that $\F(\eta)_\gamma^{-1}[1]\backslash \F(\xi)_\gamma^{-1}[1]$ is
stationary. Since $S^\kappa_\lambda$ $\diamond$-reflects to
$S_\gamma$, $A=\{\alpha\in S_\gamma\mid (\F(\eta)_\gamma^{-1}[1]\backslash
\F(\xi)_\gamma^{-1}[1])\cap \alpha=D_\alpha^\gamma\}$ is stationary and
$D_\alpha^\gamma$ is stationary in $\alpha$, and therefore $A\subseteq
F(\eta)^{-1}[1]$. On the other hand, for every $\alpha$ in $A$ we have
$\F(\xi)_\gamma^{-1}[1]\cap D_\alpha^\gamma=\emptyset$, so $A\cap
F(\xi)^{-1}[1]=\emptyset$ and we conclude that $A\subseteq
F(\eta)^{-1}[1]\Delta F(\xi)^{-1}[1]$. Therefore
$F(\eta)^{-1}[1]\Delta F(\xi)^{-1}[1]$ is stationary, and $F(\eta)$
and $F(\xi)$ are not $E_{\lambda^+\text{-club}}^ {2}$-equivalent.

Suppose $F(\eta)$ and $F(\xi)$ are not $E_{\lambda^+\text{-club}}^
{2}$-equivalent, so $F(\eta)^{-1}[1]\Delta F(\xi)^{-1}[1]$ is
stationary. Since $\lambda^+<\kappa$, by Fodor's lemma we know that
there exists $\gamma<\lambda^+$ such that $\{\alpha\in
S_\gamma\,\mid\, F(\eta)(\alpha)\neq F(\xi)(\alpha)\}$ is
stationary. Hence, the symmetric difference of $\{\alpha\in S_\gamma\mid \F(\eta)_\gamma^{-1}[1]\cap D_\alpha^\gamma\mbox{ is stationary in }\alpha\}$ and $\{\alpha\in S_\gamma\mid \F(\xi)_\gamma^{-1}[1]\cap D_\alpha^\gamma\mbox{ is stationary in }\alpha\}$ is stationary. For simplicity, let us denote by $A_\eta$ and $A_\xi$
the sets involved in this symmetric difference
(i.e. $A_\eta=\{\alpha\in S_\gamma\mid \F(\eta)_\gamma^{-1}[1]\cap
D_\alpha^\gamma\mbox{ is stationary in }\alpha\}$ and
$A_\xi=\{\alpha\in S_\gamma\mid \F(\xi)_\gamma^{-1}[1]\cap
D_\alpha^\gamma\mbox{ is stationary in }\alpha\}$). Therefore, either
$A_\eta\backslash A_\xi$ or $A_\xi\backslash A_\eta$ is
stationary. Without loss of generality we can assume that
$A_\eta\backslash A_\xi$ is stationary. Hence, $\bigcup_{\alpha\in
  A_\eta\backslash A_\xi}(\F(\eta)_\gamma^{-1}[1]\cap
D_\alpha^\gamma)\backslash \F(\xi)_\gamma^{-1}[1]$ is stationary and is
contained in $\F(\eta)_\gamma^{-1}[1]\Delta\F(\xi)_\gamma^{-1}[1]$. By Claim
2.11.2 we conclude that $\eta$ and $\xi$ are not
$E^{\k,\k}_{\lambda\text{-cub}}$-equivalent.
\end{proof}
Notice that Theorem \ref{indescr-forcing} implies the consistency
of $$E_{\lambda\text{-club}}^ {2}\ \leq_c\ E_{\lambda\text{-club}}^
{\lambda^{++}}\ \leq_c\ E_{\lambda^+\text{-club}}^ {2}\ \leq_c\
E_{\lambda^+\text{-club}}^ {\lambda^{++}}.$$ In particular, for
$\lambda=\omega$ we get the expression $E_{\o\text{-club}}^ {2}\
\leq_c\ E_{\o\text{-club}}^ {\o_2}\ \leq_c\ E_{\o_1\text{-club}}^ {2}\
\leq_c\ E_{\o_1\text{-club}}^ {\o_2}$.
\begin{Question}
  Is it consistent that $$E_{\gamma\text{-club}}^ 2\ \lneq_c\
  E_{\gamma\text{-club}}^ {\k}\ \lneq_c\ E_{\lambda\text{-club}}^
  {2}$$ holds for all $\gamma,\lambda<\k$ and $\gamma<\lambda$?
\end{Question}
We will finish this section by showing that the reduction
$E^{\omega_2}_{\omega\text{-club}}\leq_c
E^{\omega_2}_{\omega_1\text{-club}}$ can be obtained using other
reflection principles. Specifically, full reflection implies this
reduction.  For stationary subsets $S$ and $A$ of $\k$, we say that
$S$ \emph{reflects fully} in $A$ if the set $\{\a\in A\mid S\cap \a\text{ is
  non-stationary in }\a\}$ is non-stationary.  Notice that if $S\subset
S_{\gamma}^{\kappa}$ reflects fully in $S_{\lambda}^{\kappa}$, then
the set $\{\a\in S_\lambda^\k\mid S\cap \a\text{ is stationary in
}\a\}$ is a stationary set.
\begin{Thm}[\cite{JS}, Theorem 1.3]\label{JS-thm-1.3}
  Let $\kappa_2<\kappa_3<\cdots<\kappa_n<\cdots$ be a sequence of
  supercompact cardinals. There is a generic extension $V[G]$ in which
  $\kappa_n=\aleph_n$ for all $n\ge 2$ and such that:
  \begin{enumerate}
  \item Every stationary set $S\subset S_{\omega}^{\omega_2}$ reflects
    fully in $S_{\omega_1}^{\omega_2}$.
  \item For every $2< n$ and every $0\leq k\leq n-3$, every stationary
    set $S\subset S_{\omega_k}^{\omega_n}$ reflects fully in
    $S_{\omega_{n-1}}^{\omega_n}$.
  \end{enumerate}
\end{Thm}
In the generic extension of \ref{JS-thm-1.3} it holds that $\o_i^{<\o_i}=\o_i$ for all $i<\o$ (see [\cite{JS}, Theorem 1.3]).
\begin{Cor}
  The following statement is consistent:
  $E^{\omega_2}_{\omega\text{-club}}\leq_c
  E^{\omega_2}_{\omega_1\text{-club}}$, and for every $2< n$ and every
  $0\leq k\leq n-3$, $E^{\omega_n}_{\omega_k\text{-club}}\leq_c
  E^{\omega_n}_{\omega_{n-1}\text{-club}}$.
\end{Cor}
In \cite{JS} it was also proved that Theorem \ref{JS-thm-1.3} (ii) is
optimal, in the sense that it cannot be improved to include the case
$k=n-2$ \cite[Proposition 1.6]{JS}. The best possible reduction we can
get using only full reflection is the one in Corollary 2.14. By a
$\Sii$-completeness result, it is known that the following is
consistent: $$\forall k<n-1\ (E^{\omega_n}_{\omega_k\text{-club}}\
\leq_c\ E^{\omega_n}_{\omega_{n-1}\text{-club}}),$$ see Theorem 3.1
below.
\section{$\Sii$-completeness}
An equivalence relation $E$ on $X\in \{\k^\k,2^\k\}$ is \emph{$\Sii$}
if $E$ is the projection of a closed set in $X^2\times \k^\k$ and it
is \emph{$\Sii$-complete} if every $\Sii$ equivalence relation is
Borel reducible to it. The study of $\Sii$ and $\Sii$-complete
equivalence relations is an important area of generalised descriptive
set theory, because e.g.\ the isomorphism relation on classes of
models is always~$\Sii$. The same holds, in fact, in classical
descriptive set theory, but the behaviour of $\Sii$ complete relations
there is different.  For example, in the classical setting ($\k=\o$)
the isomorphism relation is never $\Sii$-complete, while in
generalised descriptive set theory this is often the case (see for
example \cite{HK,FHK}).

\begin{Thm}[\cite{HK}, Theorem 7]\label{thm:E_muSiiComplete}
  Suppose $V=L$ and $\kappa>\omega$. Then
  $E^{\k,\k}_{\mu\text{-club}}$ is $\Sii$-complete for every regular
  $\mu<\k$.
\end{Thm}
We know that $E^{\k,\k}_{\lambda\text{-cub}}\restriction\alpha$ is an
equivalence relation for every $\alpha<\kappa$ with
$cf(\alpha)>\lambda$. Let us define the following
relation: $$(\eta,\xi)\in
E_{reg}^{\k,\k}\restriction\alpha\Leftrightarrow\{\beta\in
reg(\alpha)\mid \eta(\beta)\neq \xi(\beta)\}\text{ is not
  stationary}.$$ It is easy to see that
$E_{reg}^{\k,\k}\restriction\alpha$ is an equivalence relation for
every Mahlo cardinal $\alpha<\kappa$.

\begin{Def}[$S$--dual diamond]\label{S-Dual Diamond}
  Suppose $S\subseteq\kappa$ is a stationary set. We say that
  \emph{$\k$ has the $S$--dual diamond} if: There is a sequence $\la
  f_\alpha\ra_{\alpha<\kappa}$ such that
  \begin{itemize}
  \item $f_\alpha:\alpha\rightarrow \alpha$ for all $\a$,
  \item if $(Z,g)$ is a pair such that $Z\subset S$ is stationary and
    $g\in \kappa^\kappa$, then the set $$\{\alpha\in reg(\k)\mid
    g\restriction \alpha=f_\alpha\wedge Z\cap\alpha\mbox{ is
      stationary}\}$$ is stationary.
  \end{itemize}
\end{Def}
It is clear that if $S'\supseteq S$, then the $S'$--dual diamond
implies the $S$--dual diamond.  Notice that the $S$--dual diamond has a
set version that is equivalent to it:

Suppose $S\subseteq\kappa$ is a stationary set. We say that \emph{$\k$
  has the set version $S$--dual diamond} if: There is a sequence $\la
A_\alpha\ra_{\alpha<\kappa}$ such that
\begin{itemize}
\item $A_\alpha\subseteq \alpha$ for all $\a$,
\item if $(Z, X)$ is a pair such that $Z\subset S$ is stationary and
  $X\subseteq\kappa$, then the set $$\{\alpha\in reg(\k)\mid X\cap
  \alpha=A_\alpha\wedge Z\cap\alpha\mbox{ is stationary}\}$$ is
  stationary.
\end{itemize}

It is clear, using characteristic functions, that the existence of an
$S$--dual diamond sequence in the sense of Definition \ref{S-Dual
  Diamond} implies this set version of $S$--dual diamond. For the other
implication, it is easy to check that if $\langle A_\a\rangle_{\a<
  \k}$ witnesses the set version of $D$--dual diamond, $<^\ast$ is the
canonical well order of $\k\times\k$ and $f:\k\rightarrow\k\times\k$
is the corresponding order-isomorphism, then
$B_\a=\{f(\beta)\mid\beta\in A_\a\}$ is such that: if $(Z, X)$ is a
pair such that $Z\subset S$ is stationary and $X\subseteq\k\times\k$,
then the set $$\{\alpha\in reg(\k)\mid X\cap \a\times\a=B_\a\wedge
Z\cap\a\mbox{ is stationary}\}$$ is stationary. Since every $g\in
\kappa^\kappa$ is a subset of $\k\times\k$, the sequence $\la
f_\alpha\ra_{\alpha<\kappa}$ can be constructed from the sequence $\la
B_\alpha\ra_{\alpha<\kappa}$.
\begin{Thm}\label{3.4}
  Suppose $S=S_\lambda^\k$ for some $\lambda$ regular cardinal, or
  $S=reg(\kappa)$ and $\k$ is a weakly compact cardinal. If $\k$ has
  the $S$--dual diamond, then $E^{\k,\k}_{S}\leq_c E^{2,\k}_{reg}$.
\end{Thm}
\begin{proof}
  Let $\la f_\a\ra_{\a<\k}$ be a sequence that witnesses the $S$-dual
  diamond.  Let $g_\a:\k\to\k$ be the function defined by
  $g_\a\rest\a=f_\a$ and $g_\a(\beta)=0$ for all $\beta\ge \a$. Let us
  define $F\colon \k^\k\to 2^\k$ by
  $$F(\eta)(\a)=
  \begin{cases}
    1 \text{ if }\a\in reg(\k), E_S\rest\a\text{ is an equivalence relation, and }(\eta, g_\a)\in E_S\rest\a\\
    0 \text{ otherwise}.
  \end{cases}
  $$
  
  Let us prove that if $(\eta,\xi)\in E_S$, then $(F(\eta),F(\xi))\in
  E^{2,\k}_{reg}$.  Suppose $(\eta,\xi)\in E_S$. Note that
  $F(\eta)(\a)=F(\xi)(\a)=0$ for all $\a\notin reg(\k)$, so it is
  sufficient to show that the set
  $$\{\a\in reg(\k)\mid F(\eta)(\a)\ne F(\xi)(\a)\}$$
  is non-stationary. Now, there is a club $D$ such that
  $D\cap\{\alpha\in S\mid \eta(\a)\neq\xi(\a)\}$ is
  non-stationary. So, letting $C$ be the club of the limit points of
  $D$, it holds that for all $\a\in C\cap reg(\k)$, the functions
  $\eta$ and $\xi$ are $E_S\rest\a$-equivalent. Thus, by the
  definition of $F$, at the points of the set $C\cap reg(\k)$ the
  functions $F(\eta)$ and $F(\xi)$ will get the same value.
  
  Now let us prove that if $(\eta,\xi)\notin E_S$, then
  $(F(\eta),F(\xi))\notin E^{2,\k}_{reg}$. Suppose that
  $(\eta,\xi)\notin E_S$.  Then there is a stationary $Z\subset S$ on
  which $\eta(\a)\ne\xi(\a)$.  By the definition of $S$--dual diamond,
  there is a stationary set $A\subseteq reg(\k)$ such that for all
  $\a\in A$ we have that $Z\cap \a$ is stationary and
  $\eta\rest\a=f_\a$. This means that
  $$\{\b<\a\mid \eta(\b)\ne\xi(\b)\}$$
  is stationary, and so $(\eta,\xi)\notin E_S\rest\a$ holds for all
  $\a\in A$.  However $\eta\rest\a=f_\a$ implies that $(\eta,g_\a)\in
  E_S\rest\a$, and so by transitivity $(\xi,g_\a)\notin E_S\rest\a$.
  Hence we get that $F(\eta)(\a)=1$, but $F(\xi)(\a)=0$. This holds
  for all $\a\in A$ and $A$ is stationary, so $(F(\eta),F(\xi))\notin
  E^{2,\k}_{reg}$.
\end{proof}

\begin{Lemma}\label{lem:DualDiamond}\label{weakly-compact-dual-diamond}
  Suppose $V=L$ and $\kappa$ is a weakly compact cardinal. Then $\k$
  has the $S_\o^\k$--dual diamond.
\end{Lemma}
\begin{proof}
  It is shown in the proof of \cite[Theorem 55(A)]{FHK} that $S^\k_\o$
  strongly reflects to $S^\k_{reg}$
  (Definition~\ref{def:StronglyReflects}). The rest of the proof is a
  straightforward modification of the proof of \cite[Theorem 59]{FHK},
  but we give it here for the sake of completeness. We will show that
  there is a sequence $\la D_\alpha, f_\alpha\ra_{\alpha<\kappa}$ such
  that
  \begin{itemize}
  \item $D_\alpha\subset \alpha$ is stationary in $\alpha$ for all
    $\a$,
  \item $f_\alpha:\alpha\rightarrow \alpha$ for all $\a$,
  \item if $(Z, g)$ is a pair such that $Z\subset S_\omega^\kappa$ is
    stationary and $g\in \kappa^\kappa$, then the set $$\{\alpha\in
    reg(\k)\mid g\restriction \alpha=f_\alpha\wedge
    Z\cap\alpha=D_\alpha\}$$ is stationary.
  \end{itemize}
  It is clear that this implies that $\k$ has the $S_\o^\k$-dual
  diamond.
  
  For the purpose of the proof we define a triple $\la
  D_\a,f_\a,C_\a\ra$. Suppose that $\la D_\b,f_\b,C_\b\ra$ is already
  defined for $\b<\a$.
  
  Now define $\la D, f , C\ra$ to be the $\leq_L$-least triple such
  that
  \begin{itemize}
  \item $D\subset \a\cap S^\k_\o$ is stationary,
  \item $f\colon\a\to\a$,
  \item $C$ is the intersection with $S^\k_{\reg}$ of a closed and
    unbounded subset of $\a$,
  \item for all $\b<\a$, $D\cap\b\ne D_\b\text{ or }f\rest\b\ne f_\b$,
  \end{itemize}
  and set $D_\a=D$, $f_\a=f$, $C_\a=C$ if such exists, and
  $D_\a=f_\a=C_\a=\es$ otherwise.
  
  Now our assumption is that there is a counterexample to the theorem,
  so let $(Z, g, C)$ be the $\leq_L$-least counterexample. Let $M$ be
  an elementary submodel of $L_\l$, for some regular $\l>\k$, such
  that
  \begin{itemize}
  \item $|M|<\k$,
  \item $\a=M\cap \k\in C$,
  \item $Z\cap \a$ is stationary in $\a$, and
  \item $Z$, $g$, $C$, $S^\k_\o$, $S^\k_{\reg}$, $\k\in M$.
  \end{itemize}
  $M$ exists by the $\Pii$-reflection of the weakly compact~$\k$.  Now
  take the Mostowski collapse $G\colon M\to L_\g$, for some $\g>\a$.
  Now $G(Z)=Z\cap \a$, $G(g)=g\rest\a$, $G(C)=C\cap \a$, $G(\k)=\a$,
  and the sequence $\la D_\b, f_\b\ra_{\b<\a}$ is definable in $L_\g$.
  
  Let $\f(D, g, C, \k)$ be a formula that says ``$(D, g, C)$ is the
  $\leq_L$-least triple such that
  \begin{itemize}
  \item $D\subset S^\k_\o$ is stationary,
  \item $g\colon\k\to\k$,
  \item $C$ is the intersection with $S^\k_{\reg}$ of a cub of $\k$,
    and
  \item for all $\b<\k$, $D\cap\b\ne D_\b\text{ or }f\rest\b\ne
    f_\b$.''
  \end{itemize}
  But this formula relativises to $L_\g$ and all notions are
  sufficiently absolute. When relativised, it says that $(D, g)$
  reflects to $\a\in C$, which contradicts the assumption that $(D, g,
  C)$ was a counterexample.
\end{proof}

\begin{Cor}\label{cor:E^2_reg}
  Suppose $V=L$ and $\k$ is weakly compact. Then $E^{2,\k}_{reg}$ is
  $\Sii$-complete.
\end{Cor}
\begin{proof}
  This follows from Theorem \ref{thm:E_muSiiComplete}, Lemma
  \ref{weakly-compact-dual-diamond} and Theorem \ref{3.4}.
\end{proof}

\begin{Thm}\label{3.7}
  Suppose $\k$ is a supercompact cardinal. There is a generic
  extension $V[G]$ in which $E_{reg}^{\k,\k}\leq_c E^{2,\k}_{reg}$ holds
  and where $\k$ is still supercompact.
\end{Thm}
\begin{proof}
  By Theorem \ref{3.4}, it is enough to find a forcing extension in
  which $\k$ has the $reg(\k)$--dual diamond.
  
  In \cite{Lav} it is proved that if $\kappa$  is a supercompact cardinal, then there is a forcing extension in which $\kappa$ remains supercompact upon forcing with any $\kappa$--directed closed forcing. Let us denote by $V[H]$ this forcing extension.
  
  Now we will find a forcing extension of $V[H]$ in which $\k$ has the
  $reg(\k)$--dual diamond. In fact, we will show something stronger, we
  will show that there is a forcing extension in which $\k$ has the
  $\k$--dual diamond. Working in $V[H]$, let
  $\mathbb{P}=\{f:\alpha\rightarrow\mathcal{P}(\alpha)\mid\alpha<\k\}$
  ordered by: $p\leq q$ if $q\subseteq p$. It is easy to see that
  $\mathbb{P}$ is $\k$-directed closed, and thus $\mathbb{P}\Vdash
  \check{\k} \mbox{ is supercompact}$. We will prove that $\mathbb{P}$
  forces that $\k$ has the $\k$--dual diamond. Suppose, towards a
  contradiction, that there is $G$ a $\mathbb{P}$-generic over $V[H]$
  such that $\k$ does not have the $\k$--dual diamond in $V[H][G]$. Let
  $p\in G$, $\dot{S}$, $\dot{X}$ be such that $p$ forces that the
  sequence $\{D_\alpha=(\bigcup G)(\a)\mid\alpha<\k\}$ does not guess
  $(\dot{S},\dot{X})$ as wanted, i.e.,
  
  $p\Vdash\dot{S}$\mbox{ is stationary, }$\dot{X}\subseteq \k$,
  \mbox{and the sequence }$\la D_\alpha\ra_{\a<\k}$\mbox{ does not guess
  }$\dot{X}\cap\a$ \mbox{in any }$\a$ such that $\dot{S}\cap\a$
  \mbox{is stationary.}
  
  We will show that the set $\{q<p\mid q\Vdash \exists \a\,
  (D_\a=\dot{X}\cap \a \wedge \dot{S}\cap \a\mbox{ is stationary})\}$
  is dense below $p$, which is a contradiction. There is a club
  $C\subseteq \k$ in $V[H]$ such that for all $\a\in C$ it holds that
  $p\subset \bigcup G\restriction \a$, and $\bigcup G\restriction \a$
  decides $\dot{X}\cap\a$ and $\dot{S}\cap \a$. Since $C\in V[H]$, we
  have that $C\in V[H][G]$. On the other hand, $\k$ is
  $\Pi_1^1$-indescribable in $V[H][G]$, and the sentence: $$(C \mbox{
    is unbounded in }\kappa)\wedge (\dot{S}_G \mbox{ is stationary in
  } \kappa)\wedge (\kappa \mbox{ is regular})$$ is a $\Pi^1_1$
  property of the structure $(V^{V[H][G]}_{\kappa}, \in, \dot{S}_G,
  C)$. By $\Pi_1^1$-reflection, there is $\a<\k$ in $V[H][G]$ such
  that $C\cap \a$ is unbounded, and hence $\a\in C$, $\dot{S}_G
  \cap\a$ is stationary in $\a$, and $\a$ is regular. Since
  $\mathbb{P}$ is ${<}\k$-closed, we have that $\dot{X}\cap \a\in
  V[H]$. Let $q$ be the condition $\bigcup G\restriction
  \a\cup\{(\a,\dot{X}\cap \a)\}$. Clearly $q<p$ and $q\Vdash \exists
  \a (D_\a=\dot{X}\cap \a \wedge \dot{S}\cap \a\mbox{ is stationary})$
  as we wanted.
\end{proof}

\begin{Thm}\label{3.8}
  If $\k$ is a $\Pi_2^1$-indescribable cardinal, then
  $E^{\k,\k}_{reg}$ is $\Sigma_1^1$-complete.
\end{Thm}
\begin{Remark}
  Here the notion of $\Pi^1_2$--indescribability is the usual one, not to be confused with the $\Pi^\lambda_1$--indescribability from Definition \ref{def:PiLambda1Indescr}.
\end{Remark}
\begin{proof}
  Let $E$ be a $\Sii$ equivalence relation on $\kappa^\kappa$. Then
  there is a closed set $C$ on $\k^\k\times\k^\k\times\k^\k$ such that
  $\eta\ E\ \xi$ if and only if there exists $\theta\in \k^\k$ such
  that $(\eta,\xi,\theta)\in C$. Let us define $U=\{(\eta\restriction
  \alpha, \xi\restriction \alpha, \theta\restriction \alpha)\,\mid\,
  (\eta, \xi, \theta)\in C\wedge\a<\k\}$, and for every $\gamma<\k$
  define $C_\gamma=\{(\eta, \xi,
  \theta)\in\gamma^\gamma\times\gamma^\gamma\times\gamma^\gamma\,\mid\,\forall\a<\gamma\
  (\eta\restriction \alpha, \xi\restriction \alpha, \theta\restriction
  \alpha)\in U\}$. Let $E_\gamma\subset
  \gamma^\gamma\times\gamma^\gamma$ be the relation defined by
  $(\eta,\xi)\in E_\gamma$ if and only if there exists $\theta\in
  \gamma^\gamma$ such that $(\eta,\xi,\theta)\in C_\gamma$. Since $E$
  is an equivalence relation, it follows that $E_\gamma$ is reflexive
  and symmetric, but not necessary transitive. Let us define the
  reduction by
  $$F(\eta)(\a)=
  \begin{cases}
    f_\a(\eta) \text{ if }E_\a\text{ is an equivalence relation and }\eta\restriction\a\in\a^\a\\
    0 \text{ otherwise}.
  \end{cases}
  $$
  where $f_\a(\eta)$ is a code in $\k\backslash\{0\}$ for the
  $E_\a$-equivalence class of $\eta$.
  
  Let us prove that if $(\eta, \xi)\in E$, then $(F(\eta), F(\xi))\in
  E^{\k,\k}_{reg}$. Suppose $(\eta, \xi)\in E$. Then there is
  $\theta\in \k^\k$ such that $(\eta, \xi, \theta)\in C$ and for all
  $\a<\k$ we have that $(\eta\restriction \alpha, \xi\restriction
  \alpha, \theta\restriction \alpha)\in U$. On the other hand, we know
  that there is a club $D$ such that for all $\a\in D\cap reg(\k)$,
  $\eta\restriction\a$, $\xi\restriction\a$,
  $\theta\restriction\a\in\a^\a$. We conclude that for all $\a\in
  D\cap reg(\k)$, if $E_\a$ is an equivalence relation, then
  $(\eta,\xi)\in E_\a$. Therefore, for all $\a\in D\cap reg(\k)$,
  $F(\eta)(\a)=F(\xi)(\a)$, so $(F(\eta),F(\xi))\in E^{\k,\k}_{reg}$.
  Let us prove that if $(\eta,\xi)\notin E$, then
  $(F(\eta),F(\xi))\notin E^{\k,\k}_{reg}$. Suppose $\eta$, $\xi\in
  \k^\k$ are such that $(\eta, \xi)\notin E$. We know that there is a
  club $D$ such that for all $\a\in D\cap reg(\k)$,
  $\eta\restriction\a$, $\xi\restriction\a\in\a^\a$.
  
  Notice that because $C$ is closed $(\eta, \xi)\notin E$ is
  equivalent to $$\forall\theta\in \k^\k\ (\exists\a<\k\
  (\eta\restriction \alpha,\xi\restriction \alpha,\theta\restriction
  \alpha)\notin U),$$ so the sentence $(\eta,\xi)\notin E$ is a
  $\Pi_1^1$ property of the structure $(V_\k,\in,U,\eta,\xi)$. On the
  other hand, the sentence $\forall\zeta_1,\zeta_2,\zeta_3\in \k^\k
  [((\zeta_1,\zeta_2)\in E\wedge(\zeta_2,\zeta_3)\in
  E)\rightarrow(\zeta_1,\zeta_3)\in E]$ is equivalent to the sentence
  $\forall \zeta_1,\zeta_2,\zeta_3,\theta_1,\theta_2\in \k^\k
  [\exists\theta_3\in \k^\k(\psi_1\vee\psi_2\vee\psi_3)]$, where
  $\psi_1$, $\psi_2$ and $\psi_3$ are, respectively, the formulas
  $\exists\alpha_1<\k\ (\zeta_1\restriction
  \alpha_1,\zeta_2\restriction \alpha_1,\theta_1\restriction
  \alpha_1)\notin U$, $\exists\alpha_2<\k\ (\zeta_2\restriction
  \alpha_2,\zeta_3\restriction \alpha_2,\theta_2\restriction
  \alpha_2)\notin U$, and $\forall\alpha_3<\k\ (\zeta_1\restriction
  \alpha_3,\zeta_3\restriction \alpha_3,\theta_3\restriction
  \alpha_3)\in U$. Therefore, the sentence
  $\forall\zeta_1,\zeta_2,\zeta_3\in \k^\k [((\zeta_1,\zeta_2)\in
  E\wedge(\zeta_2,\zeta_3)\in E)\rightarrow(\zeta_1,\zeta_3)\in E]$ is
  a $\Pi_2^1$ property of the structure $(V_\k,\in,U)$. It follows
  that the sentence $$(D \mbox{ is unbounded in
  }\k)\wedge((\eta,\xi)\notin E)\wedge (E\mbox{ is an equivalence
    relation})\wedge (\k\mbox{ is regular})$$ is a $\Pi_2^1$ property
  of the structure $(V_\k,\in,U,\eta,\xi)$. By $\Pi_2^1$ reflection,
  we know that there are stationary many $\gamma\in reg(\k)$ such that
  $\gamma$ is a limit point of $D$, $E_\gamma$ is an equivalence
  relation, and $(\eta\restriction \gamma,\xi\restriction
  \gamma)\notin E_\gamma$. We conclude that there are stationary many
  $\gamma\in reg(\k)$ such that $f_\gamma(\eta)\neq f_\gamma(\xi)$,
  and hence $(F(\eta),F(\eta))\notin E^{\k,\k}_{reg}$.
\end{proof}

\begin{Cor}\label{cor2:E^2_reg}
  Suppose $\k$ is a supercompact cardinal. There is a generic
  extension $V[G]$ in which $E_{reg}^{2,\k}$ is $\Sii$-complete.
\end{Cor}
\begin{proof}
  Let $V[G]$ be the generic extension of Theorem \ref{3.7}. Since $\k$
  is supercompact in $V[G]$, it is $\Pi_2^1$-indescribable. By Theorem
  \ref{3.8}, $E_{reg}^\k$ is $\Sii$-complete, and by Theorem \ref{3.7}
  we know that $E_{reg}^\k\leq_c E_{reg}^{2,\k}$. We conclude that
  $E_{reg}^{2,\k}$ is $\Sii$-complete in $V[G]$.
\end{proof}
Let $NS$ denote the equivalence on $2^\k$ modulo the non-stationary
ideal, i.e. $\eta\ NS\ \xi$ if and only if $\eta^{-1}[1]\triangle
\xi^{-1}[1]$ is not stationary. For every stationary $S\subseteq \k$
the relation $E^{2,\k}_S$ is continuously reducible to $NS$. The
reduction $\mathcal{F}:2^\k\rightarrow 2^\k$ is defined as
follows: $$\mathcal{F}(\eta)(\a)=
\begin{cases}
  \eta(\a) &\mbox{if } \a\in S\\
  1 & \mbox{otherwise }
\end{cases}
$$
We conclude that the statement \textit{$NS$ is $\Sii$-complete} is
consistent, this follows from Corollary \ref{cor:E^2_reg} (it also
follows from Corollary \ref{cor2:E^2_reg}).

We will finish this article with a result related to model theory.

\begin{Thm}\label{thm:DLOatPi21}
  Let DLO be the theory of dense linear orderings without end
  points. If $\k$ is a $\Pi_2^1$-indescribable cardinal, then
  $\cong_{\DLO}$ is $\Sii$-complete.
\end{Thm}
\begin{proof}
  By Theorem \ref{3.8} it is enough to show that $E_{reg}^\k\leq_c
  \cong_{\DLO}$. To show this, first we will construct models of
  DLO, $\mathcal{A}^{\mathcal{F}(f)}$, for every $f:\k\rightarrow \k$,
  such that $f\ E_{reg}^\k\ g$ if and only if
  $\mathcal{A}^{\mathcal{F}(f)}\cong
  \mathcal{A}^{\mathcal{F}(g)}$. After that we construct the reduction
  of $E_{reg}^\k$ to $\cong_{\DLO}$.
  
  Let us take the language $\mathcal{L}'=\{L,C,<,R\}$, with $L$ and
  $C$ as unary predicates, and $<$ and $R$ as binary relations. Let
  $K$ be the class of $\mathcal{L}'$-structures
  $\mathcal{A}=(dom(\mathcal{A}),L,C,<,R)$ that satisfy the following
  conditions:
  \begin{itemize}
  \item $L\cap C=\emptyset$.
  \item $L\cup C=dom(\mathcal{A})$.
  \item $<\ \subseteq L\times L$ is a dense linear order without end
    points on $L$.
  \item $R\subseteq L\times C$.
  \item Let us denote by $R^-(y,x)$ the formula $\neg R(y,x)$. For all
    $x\in C$, it holds that $R(\mathcal{A},x)\cup
    R^-(\mathcal{A},x)=L$, $R(\mathcal{A},x)$ has no largest element,
    and $R^-(\mathcal{A},x)$ has no least element and they are
    non-empty.
  \end{itemize}
  Let us define the following partial order $\preceq$ on $K$. We say
  that $\mathcal{A}\preceq \mathcal{B}$ iff:
  \begin{itemize}
  \item $\mathcal{A}\subseteq \mathcal{B}$,
  \item for all $x\in C^\mathcal{A}$, $R(\mathcal{B},x)=\{y\in
    L^\mathcal{B}\mid\exists z\in R(\mathcal{A},x), y<z\}$ and
    $R^-(\mathcal{B},x)=\{y\in L^\mathcal{B}\mid\exists z\in
    R^-(\mathcal{A},x), z<y\}$,
  \item for all $x\in C^\mathcal{B}\backslash C^\mathcal{A}$ there are
    $y\in R(\mathcal{B},x)$ and $z\in R^-(\mathcal{B},x)$ such that
    for all $a\in L^\mathcal{A}$, $a<y\vee a>z$.
  \end{itemize}
  Notice that it is possible to have a chain
  $\mathcal{A}_0\preceq\mathcal{A}_1\preceq\cdots$ of length $\a$ in
  $K$, and a structure $\mathcal{C}\in K$, such that
  $\bigcup_{i<\a}\mathcal{A}_i\in K$, $\mathcal{A}_i\preceq
  \mathcal{C}$ holds for all $i<\a$, and
  $\bigcup_{i<\a}\mathcal{A}_i\not\preceq \mathcal{C}$. But all other
  requirements of AEC's are satisfied, as one can easily see, in
  particular for every chain
  $\mathcal{A}_0\preceq\mathcal{A}_1\preceq\cdots$ of length $\a$ in
  $K$, $\bigcup_{i<\a}\mathcal{A}_i\in K$.
  \begin{Claim}\label{3.10.1}
    $(K,\preceq)$ has the amalgamation property and the joint
    embedding property.
  \end{Claim}
  \begin{proof}
    The joint embedding property is easily seen to follow from the
    amalgamation property. For the amalgamation property, let
    $\mathcal{A},\mathcal{B},\mathcal{C}\in K$ be such that
    $\mathcal{A}\preceq \mathcal{B}$ and $\mathcal{A}\preceq
    \mathcal{C}$ hold. Without loss of generality, we can assume that
    $dom(\mathcal{B})\cap dom(\mathcal{C})=dom(\mathcal{A})$. Let us
    construct $\mathcal{D}$ with $\dom(\mathcal{B})\cup
    dom(\mathcal{C})=dom(\mathcal{D})$,
    $L^{\mathcal{D}}=L^{\mathcal{B}}\cup L^{\mathcal{C}}$, and
    $C^{\mathcal{D}}=C^{\mathcal{B}}\cup C^{\mathcal{C}}$. To define
    $<^{\mathcal{D}}$ and $R^{\mathcal{D}}$, first define
    $<'=<^{\mathcal{B}}\cup <^{\mathcal{C}}$. For every two elements
    $b,c\in L^{\mathcal{D}}$ define $b<^{\mathcal{D}}c$ if either
    $b<'c$, or there is $a\in L^{\mathcal{A}}$ such that $b<'a<'c$, or
    $b\in L^{\mathcal{B}}$, $c\in L^{\mathcal{C}}$ and there is no
    $a\in L^{\mathcal{A}}$ such that $c<'a<'b$. For every $x\in
    C^\mathcal{A}$, $R(\mathcal{D},x)=R(\mathcal{B},x)\cup
    R(\mathcal{C},x)$. For all $x\in C^\mathcal{B}\backslash
    C^\mathcal{A}$, $y\in R(\mathcal{D},x)$ if and only if there
    exists $z\in L^\mathcal{B}$ such that $z\in R(\mathcal{B},x)$ and
    $y<^\mathcal{D}z$. For all $x\in C^\mathcal{C}\backslash
    C^\mathcal{A}$, $y\in R(\mathcal{D},x)$ if and only if there
    exists $z\in L^\mathcal{C}$ such that $z\in R(\mathcal{C},x)$ and
    $y<^\mathcal{D}z$. It is clear that $\mathcal{D}\in K$, and
    $\mathcal{B}\preceq\mathcal{D}$ and
    $\mathcal{C}\preceq\mathcal{D}$.
  \end{proof}
  Let us denote by $\mathcal{A}_1\oplus_{\mathcal{A}_0}\mathcal{A}_2$
  the structure $\mathcal{D}$, in Claim \ref{3.10.1}, that witnesses
  the amalgamation property for the structures
  $\mathcal{A}_0\preceq\mathcal{A}_1$ and
  $\mathcal{A}_0\preceq\mathcal{A}_2$. For every ordinal $\a$, let us
  denote by $\a^*$ the set $\a$ ordered by the reverse order $<^*$,
  i.e., $\beta<^*\gamma$ if $\gamma\in \beta$. Let us order the
  members of $\mathbb{Q}\times\a^*$ by: $(r_1,\a_1)<^{*\a}(r_2,\a_2)$
  iff $\a_1<^*\a_2$, or $\a_1=\a_2$ and $r_1<^\mathbb{Q} r_2$.
  
  Let $K_{{<}\kappa}$ be the collection of all members of $K$ of size less than $\kappa$. For every $\mathcal{A}\in K_{<\k}$, denote by
  $\{\mathcal{A}(i)\}_{i<\k}$ an enumeration of all the strong
  extensions of $\mathcal{A}$, i.e. $\mathcal{A}\preceq \mathcal{B}$,
  of size less than $\k$ (up to isomorphism over $\mathcal{A}$). Let
  $\Pi:\k\rightarrow\k\times\k$, $\Pi(\a)=(pr_1(\Pi(\a)),
  pr_2(\Pi(\a)))$ be a bijection such that $pr_1(\Pi(i))\leq i$ for
  all $i$. Given a function $f:\k\rightarrow reg(\k)$, let us
  construct the following sequence of models:
  \begin{itemize}
  \item $\mathcal{A}^f_0=(\mathbb{Q},\emptyset,<,\emptyset)$.
  \item For a successor ordinal, let
    $\mathcal{D}=\mathcal{A}^f_i\oplus_{\mathcal{A}^f_{pr_1(\Pi(i))}}\mathcal{A}^f_{pr_1(\Pi(i))}(pr_2(\Pi(i)))$. Define
    $L^{\mathcal{A}^f_{i+1}}=L^\mathcal{D}\cup\mathbb{Q}$,
    $C^{\mathcal{A}^f_{i+1}}=C^\mathcal{D}$,
    $<^{\mathcal{A}^f_{i+1}}=<^\mathcal{D}\cup <^\mathbb{Q}\cup
    \{(x,y)\mid x\in L^\mathcal{D}\wedge y\in \mathbb{Q}\}$, and
    $R^{\mathcal{A}^f_{i+1}}=R^\mathcal{D}$. Clearly
    $\mathcal{A}^f_{i+1}\in K$.
  \item For $i$ a limit ordinal, let
    $\mathcal{D}=\bigcup_{j<i}\mathcal{A}^f_j$. Define
    $L^{\mathcal{A}^f_{i}}=L^\mathcal{D}\cup(\mathbb{Q}\times
    f(i)^*)$, $C^{\mathcal{A}^f_{i}}=C^\mathcal{D}\cup \{x\}$,
    $<^{\mathcal{A}^f_{i}}=<^\mathcal{D}\cup <^{*f(i)}\cup \{(a,b)\mid
    a\in L^\mathcal{D}\wedge b\in \mathbb{Q}\times f(i)^*\}$, and
    $R^{\mathcal{A}^f_{i}}=R^\mathcal{D}\cup \{(y,x)\mid y\in
    L^\mathcal{D}\}$. Clearly $\mathcal{A}^f_{i}\in K$.
  \end{itemize}
  Define $\mathcal{A}^f_\k$ by $\bigcup_{j<\k}\mathcal{A}^f_j$. Then
  $\mathcal{A}^f=(L^{\mathcal{A}^f_\k},<^{\mathcal{A}^f_\k})$ is a
  model of DLO.
  
  Notice that if $i<\k$ and $\mathcal{C}\in K$,
  $\arrowvert\mathcal{C}\arrowvert<\k$, are such that
  $\mathcal{A}^f_i\preceq \mathcal{C}$, then there is $j<\k$ such that
  $\mathcal{A}^f_i(j)=\mathcal{C}$. Therefore there is $l<\k$ such
  that $\Pi(l)=(i, j)$,
  $\mathcal{A}^f_{pr_1(\Pi(l))}=\mathcal{A}^f_i$, and
  $\mathcal{A}^f_{pr_1(\Pi(l))}(pr_2(\Pi(l)))=\mathcal{C}$. We
  conclude that if $i<\k$ and $\mathcal{C}\in K_{<\k}$ are such that
  $\mathcal{A}^f_i\preceq \mathcal{C}$, then there is $j<\k$ and a
  strong embedding $F:\mathcal{C}\rightarrow \mathcal{A}^f_j$ such
  that $F(\mathcal{C})\preceq\mathcal{A}^f_j$ and $F\restriction
  \mathcal{A}^f_i=id$. Now we will show that if $f$ and $g$ are
  functions from $\k$ into $reg(\k)$ such that
  $f\restriction(\k\backslash reg(\k))=g\restriction(\k\backslash
  reg(\k))$, then $f\ E_{reg}^\k\ g$ if and only if
  $\mathcal{A}^f\cong \mathcal{A}^g$. First of all, let us prove that
  $(f, g)\in E^{\k,\k}_{reg}$ implies $\mathcal{A}^f\cong
  \mathcal{A}^g$. Suppose $(f, g)\in E^{\k,\k}_{reg}$. Then there is a
  club $C$ such that for all $\a\in C\cap reg(\k)$,
  $f(\a)=g(\a)$. Since $f\restriction(\k\backslash
  reg(\k))=g\restriction(\k\backslash reg(\k))$, we have that for all
  $\a\in C$, $f(\a)=g(\a)$. By the way the models $\mathcal{A}^f_\a$
  and $\mathcal{A}^g_\a$ were constructed for $\a$ a limit ordinal, we
  know that if $\a$ is such that $f(\a)=g(\a)$ and there is an
  isomorphism $F:\bigcup_{i<\a}\mathcal{A}^f_i\rightarrow
  \bigcup_{i<\a}\mathcal{A}^g_i$, then there is an isomorphism
  $G:\mathcal{A}^f_\a\rightarrow \mathcal{A}^g_\a$ such that
  $F\subseteq G$. For all $i<\k$ construct $\a_i<\k$ and a strong
  embedding $F_i$ such that the following hold:
  \begin{enumerate}
  \item For every $i<\k$ there is some $\gamma\in C$ such that
    $\a_i<\gamma<\a_{i+1}$.
  \item For all $i<j<\k$, $f_i\subseteq f_j$.
  \item The following holds for every limit ordinal $\beta<\k$:
    \begin{itemize}
    \item for every even $0<i<\o$,
      $dom(F_{\beta+i})=\mathcal{A}^f_{\a_{\beta+i}}$, and
      $F_{\beta+i}(\mathcal{A}^f_{\a_{\beta+i}})\preceq\mathcal{A}^g_{\a_{\beta+i+1}}$,
    \item for every odd $0<i<\o$,
      $rang(F_{\beta+i})=\mathcal{A}^g_{\a_{\beta+i}}$, and
      $F^{-1}_{\beta+i}(\mathcal{A}^g_{\a_{\beta+i}})\preceq\mathcal{A}^f_{\a_{\beta+i+1}}$,
    \item for $i=0$, $\a_\beta=\bigcup_{i<\beta}\a_i$,
      $dom(F_{\beta})=\mathcal{A}^f_{\a_{\beta}}$, and
      $rang(F_{\beta})=\mathcal{A}^g_{\a_{\beta}}$.
    \end{itemize}
  \end{enumerate}
  We will construct these sequences by induction. For $i=0$, take
  $\a_0=0$ and $F_0=id$.
  
  Successor case: Suppose $\beta$ is a limit ordinal or zero, and
  $0\leq i<\o$ are such that $\a_{\beta+i}$ and $F_{\beta+i}$ are
  constructed such that (i), (ii), and (iii) are satisfied. Let us
  start with the case when $i$ is odd. Choose $\a_{\beta+i+1}$ such
  that (i) holds. Since $F^{-1}(\mathcal{A}^g_{\a_{\beta+i}})\preceq
  \mathcal{A}^f_{\a_{\beta+i+1}}$, there are $\mathcal{C}\in K_{<\k}$
  and $F\supseteq F_{\beta+i}$ such that
  $\mathcal{A}^g_{\a_{\beta+i}}\preceq \mathcal{C}$ and
  $F:\mathcal{A}^f_{\a_{\beta+i+1}}\rightarrow \mathcal{C}$ is an
  isomorphism. By the observation we made above, there is $j<\k$ and a
  strong embedding $G:\mathcal{C}\rightarrow \mathcal{A}^g_j$ such
  that $G(\mathcal{C})\preceq\mathcal{A}^g_j$ and $G\restriction
  \mathcal{A}^g_{\a_{\beta+i}}=id$. Define $F_{\a_{\beta+i+1}}=G\circ
  F_{\a_{\beta+i}}$. Clearly $F_{\a_{\beta+i+1}}$ satisfies conditions
  (ii) and (iii). The case when $i$ is even is similar to the odd
  case.
  
  Limit case: Suppose $\beta$ is a limit ordinal such that for all
  $i<\beta$, $\a_{i}$ and $F_{i}$ are constructed such that (i), (ii),
  and (iii) are satisfied. By (i), we know that
  $\a_\beta=\bigcup_{i<\beta}\a_i$ is a limit point of $C$, so
  $f(\a_\beta)=g(\a_\beta)$. On the other hand, by conditions (ii) and
  (iii) we know
  that $$\bigcup_{i<\beta}F_i:\bigcup_{i<\beta}\mathcal{A}^f_{\a_i}\rightarrow\bigcup_{i<\beta}\mathcal{A}^g_{\a_i}$$
  is an isomorphism. Therefore, there is an isomorphism
  $G:\mathcal{A}^f_\a\rightarrow \mathcal{A}^g_\a$ such that
  $\bigcup_{i<\beta}F_i\subseteq G$. We conclude that $F_{\a_\beta}=G$
  satisfies (ii) and (iii).
  
  Finally, notice
  that $$\bigcup_{i<\k}F_i:\bigcup_{i<\k}\mathcal{A}^f_{\a_i}\rightarrow\bigcup_{i<\k}\mathcal{A}^g_{\a_i}$$
  is an isomorphism. We conclude that $\mathcal{A}^f$ and
  $\mathcal{A}^g$ are isomorphic.
  
  Let us prove that $\mathcal{A}^f\cong \mathcal{A}^g$ implies
  $(f,g)\in E^{\k,\k}_{reg}$. Suppose, towards a contradiction, that
  $(f,g)\notin E^{\k,\k}_{reg}$ and there is an isomorphism
  $F:\mathcal{A}^f\rightarrow \mathcal{A}^g$. Since $F$ is an
  isomorphism, there is a club $C$ such that
  $F(\bigcup_{i<\a}\mathcal{A}^f_{i})=\bigcup_{i<\a}\mathcal{A}^g_{i}$
  holds for all $\a\in C$. Since $(f, g)\notin E^{\k,\k}_{reg}$,
  $C\cap \{\a\in reg(\k)\mid f(\a)\neq g(\a)\}$ is nonempty. Take
  $\a\in C\cap \{\gamma\in reg(\k)\mid f(\gamma)\neq g(\gamma)\}$. We
  know that
  $F(\bigcup_{i<\a}\mathcal{A}^f_{i})=\bigcup_{i<\a}\mathcal{A}^g_{i}$
  and $f(\a)\neq g(\a)$. Hence, the co-initiality of $\{a\in
  \mathcal{A}^f\mid \forall b\in \bigcup_{i<\a}\mathcal{A}^f_i
  (b<^{\mathcal{A}^f}a)\}$ with respect to $<^{\mathcal{A}^f}$ is
  $f(\a)$. Since $F$ is an isomorphism and
  $F(\bigcup_{i<\a}\mathcal{A}^f_{i})=\bigcup_{i<\a}\mathcal{A}^g_{i}$,
  the co-initiality of $\{a\in \mathcal{A}^g\mid \forall b\in
  \bigcup_{i<\a}\mathcal{A}^g_i (b<^{\mathcal{A}^g}a)\}$ with respect
  to $<^{\mathcal{A}^g}$ is also $f(\a)$. We conclude that
  $f(\a)=cf(g(\a))$, so $f(\a)=g(\a)$, a contradiction.  To finish
  with the construction of the models, let us define
  $\mathcal{A}^{\mathcal{F}(f)}$ for all $f:\k\rightarrow \k$. Fix a
  bijection $G:\k\rightarrow reg(\k)$. Define
  $\mathcal{F}:\k^\k\rightarrow\k^\k$ by $$\mathcal{F}(f)(\a)=
  \begin{cases}
    G(f(\a)) &\mbox{if } \a\in reg(\k)\\
    0 & \mbox{otherwise }
  \end{cases}
  $$
  Clearly $f\ E^{\k,\k}_{reg}\ g$ if and only if $\mathcal{F}(f)\
  E^{\k,\k}_{reg}\ \mathcal{F}(g)$, and $\mathcal{F}(f)\
  E^{\k,\k}_{reg}\ \mathcal{F}(g)$ if and only if
  $\mathcal{A}^{\mathcal{F}(f)}$ and $\mathcal{A}^{\mathcal{F}(g)}$
  are isomorphic.  Now we will construct a reduction of
  $E^{\k,\k}_{reg}$ to $\cong_{\DLO}$ by coding the models
  $\mathcal{A}^{\mathcal{F}(f)}$ by functions $\eta:\k\rightarrow\k$.
  
  Clearly the models $\mathcal{A}^{\mathcal{F}(f)}$ satisfy
  that $$\mathcal{F}(f)\restriction\alpha=\mathcal{F}(g)\restriction\alpha\Leftrightarrow
  \mathcal{A}^{\mathcal{F}(f)}_\alpha=\mathcal{A}^{\mathcal{F}(g)}_\alpha.$$
  For every $f\in \kappa^\kappa$ define $C_f\subseteq Card\cap\kappa$
  such that for all $\alpha\in C_f$, it holds that for every
  $\beta<\alpha$,
  $\arrowvert\mathcal{A}_\beta^{\mathcal{F}(f)}\arrowvert<\arrowvert\mathcal{A}_\alpha^{\mathcal{F}(f)}\arrowvert$. For
  every $f\in \kappa^\kappa$ and $\alpha\in C_f$ choose a bijection
  $E^\alpha_f:dom(\mathcal{A}_\alpha^{\mathcal{F}(f)})\rightarrow\arrowvert\mathcal{A}_\alpha^{\mathcal{F}(f)}\arrowvert$
  such that for all $\beta<\alpha$ in $C_f$ it holds that
  $E_f^\beta\subseteq E_f^\alpha$. Then $\bigcup_{\alpha\in
    C_f}E_f^\alpha=E_f$ is such that
  $E_f:dom(\mathcal{A}^{\mathcal{F}(f)})\rightarrow\kappa$ is a
  bijection, and for every $f,g\in \kappa^\kappa$ and $\alpha<\kappa$
  the following holds: If
  $\mathcal{F}(f)\restriction\alpha=\mathcal{F}(g)\restriction\alpha$,
  then $E_f\restriction
  dom(\mathcal{A}^{\mathcal{F}(f)}_\alpha)=E_g\restriction
  dom(\mathcal{A}^{\mathcal{F}(g)}_\alpha)$.
  Let $\pi$ be the bijection in Definition 1.6. Define the function 
  $\mathcal{G}$ by: 

  $$\mathcal{G}(\mathcal{F}(f))(\alpha)=
  \begin{cases} 
    1 &\mbox{if } \alpha=\pi(m,a_1,\ldots,a_n) \text{ and } 
    \mathcal{A}^{\mathcal{F}(f)}\models P_m(E_f^{-1}(a_1),\ldots,E_f^{-1}(a_n))\\
    0 & \mbox{in the other case. }
  \end{cases}
  $$

  To show that $\mathcal{G}$ is continuous, let
  $[\eta\restriction\alpha]$ be a basic open set and $\xi\in
  \mathcal{G}^{-1}[[\eta\restriction\alpha]]$. There is $\beta\in
  C_\xi$ such that for all $\gamma<\alpha$, if
  $\gamma=\pi(m,a_1,a_2,\ldots,a_n)$, then $E^{-1}_\xi(a_i)\in
  dom(\mathcal{A}^\xi_\beta)$ holds for all $i\leq n$. Since for all
  $\zeta\in [\xi\restriction\beta]$ it holds that
  $\mathcal{A}^\xi_\beta=\mathcal{A}^\zeta_\beta$, for every
  $\gamma<\alpha$ such that $\gamma=\pi(m,a_1,a_2,\ldots,a_n)$, it
  holds that
  $$\mathcal{A}^{\xi}\models P_m(E_\xi^{-1}(a_1),E_\xi^{-1}(a_2),\ldots,E_\xi^{-1}(a_n))$$ if and only if $$\mathcal{A}^{\zeta}\models P_m(E_\zeta^{-1}(a_1),E_\zeta^{-1}(a_2),\ldots,E_\zeta^{-1}(a_n))$$
  We conclude that $\mathcal{G}(\zeta)\in [\eta\restriction\alpha] $,
  and $\mathcal{G}\circ\mathcal{F}$ is a continuous reduction of
  $E^{\k,\k}_{reg}$ to $\cong_{\DLO}$.
\end{proof}
\providecommand{\bysame}{\leavevmode\hbox to3em{\hrulefill}\thinspace}
\providecommand{\MR}{\relax\ifhmode\unskip\space\fi MR }
\providecommand{\MRhref}[2]{%
  \href{http://www.ams.org/mathscinet-getitem?mr=#1}{#2} }
\providecommand{\href}[2]{#2}

\end{document}